\documentclass[11pt]{article}
\usepackage{amsfonts,amsmath,amsthm, mathtools, bm, xspace}
\usepackage{multirow,graphicx, algorithm,algorithmic,rotating, xcolor}
\usepackage{url,cite}
\usepackage{fullpage}
\usepackage[small,bf]{caption}
\usepackage{hyperref}

\newtheorem{theorem}{Theorem}[section]
\newtheorem{lemma}[theorem]{Lemma}

\newtheorem{proposition}[theorem]{Proposition}
\newtheorem{definition}[theorem]{Definition}

\newtheorem{remark}[subsection]{Remark}

\newtheorem{example}[subsection]{Example}

\newcommand{\eat}[1]{}

\numberwithin{equation}{section}

\title{Online Control For Adaptive Tapering Of Medications}

\author{Paula Gradu\footnote{EECS, UC Berkeley. E-mail: {\tt pgradu@berkeley.edu}}\; and Benjamin Recht\footnote{EECS, UC Berkeley. E-mail: {\tt brecht@berkeley.edu.}}}

\date{August 2023}

\begin{document}

\maketitle

\vspace{-0.3in}

\begin{abstract} \noindent We investigate adaptive protocols for the elimination or reduction of the use of medications or addictive substances. We formalize this problem as online optimization, minimizing the cumulative dose subject to constraints on well-being. We adapt a model of addiction from the psychology literature and show how it can be described by a class of linear time-invariant systems. For such systems, the optimal policy amounts to taking the smallest dose that maintains well-being. We derive a simple protocol based on integral control that requires no system identification, only needing approximate knowledge of the instantaneous dose response. This protocol is robust to model misspecification and is able to maintain an individual's well-being during the tapering process. Numerical experiments demonstrate that the adaptive protocol outperforms non-adaptive methods in terms of both maintenance of well-being and rate of dose reduction.
\end{abstract}

%{\bf Keywords.} adaptive tapering, online control

\section{Introduction}
Tapering medications and assisting cessation of addictive substances are related challenges in health care. Relapse rates are  high, and many return to drug use within weeks of entering treatment. Though no general protocols are equally valuable to all medications or all people, a personalized approach to cessation may allow caregivers flexibility to meet the diverse needs of a diverse care-seeking population. This paper investigates adaptive tapering protocols that enable individuals to self-regulate their cessation rate. In particular, we formalize tapering as an optimization problem of minimizing the cumulative dose subject to constraints on well-being. 

Such a formulation requires modeling an individual's dose-response dynamics, and we propose a mathematical formulation of the \emph{opponent process} model of addiction due to Solomon \cite{opponent_process}. As we discuss in Section~\ref{sec:related}, an opponent process consists of two competing systemic reactions to treatment. The first system governs the ``positive'' effect with a rapid onset and fast subsequent decay. The second system governs ``negative effects'' with a long latency and slow decay. In Section~\ref{sec:prelim}, we show that modeling opponent processes as LTI systems captures the qualitative behaviors specified in the psychology literature and allows us to analyze tapering as an optimal control problem.

In Section~\ref{sec:opt-result}, we first derive a greedy policy that solves the optimal control problem. The dosage should be decreased by as much as possible while maintaining constraints on well-being, but no planning is needed to compute the dose. We provide a robust controller that maintains this greedy behavior and is optimal under various models of possible exogenous disturbances.

In reality, the opponent process model is an approximation, and the particular model of an individual is unknown. A possible approach informed by control theoretic practice might involve learning a dynamics model of the individual, but such learning procedures require wide varying of inputs and would not be feasible for patients attempting drug cessation. Instead of leaning on system identification, we investigate the potential of ``model-free'' controllers for tapering in Section~\ref{sec:integral}. We analyze the potential of integral control, where the error signal is the deviation from a minimally acceptable level of well-being. Using methods from the online learning literature, we show that integral control robustly reduces the dosage while minimally violating constraints on well-being. Under further mild assumptions about the time between doses, we show that the integral controller monotonically reduces the dose to zero in finite time without violating the specified constraints. In the numerical experiments of Section~\ref{sec:experiments}, we demonstrate that our adaptive protocol outperforms non-adaptive methods in terms of both maintenance of well-being and rate of dose reduction.
%We conclude with potential further investigations in Section~\ref{sec:discussion}.

\section{Related Work}\label{sec:related}
\paragraph{Tapering Protocols.} Our work attempts to synthesize adaptive tapering protocols that can be applied both for de-prescribing medications and for assisting cessation of addictive substances. In both of these applications, most studies of tapering protocols have focused on non-personalized substance-specific procedures. 

For example, for discontinuing SSRI treatment, Horowitz and Taylor \cite{horowitz2019tapering} claim that standard recommendations for tapering plans are too rapid, and better results are achieved if tapering is performed over multiple months than over weeks. They find that recommendations for linear dose reductions by constant amounts compounds these negative effects. To address this issue, they suggest that SSRI's be tapered via exponentially decreasing regimens, reducing the dose by multiplicative rather than additive factors. They also highlight the importance of individualisation of the process, despite there being no set approach for doing so. 

Similarly, withdrawal guidelines for benzodiazepines recommend dose reductions that are proportional to the present dose (most commonly 10\% reductions) per week, yielding exponentially decreasing regimens, as opposed to linear reductions.
Horowitz et al. \cite{antipsychotic_taper, antipsychotic_taper_schizo} propose a similar method (very slow, hyperbolic) for tapering of antipsychotic medication. This work is based on case studies and notes that there is no standard guideline for tapering antipsychotic medications. 

Several studies have also investigated protocols for managing withdrawal from addictive substances. In a meta-analysis of opiod tapering protocols, Berna et al. \cite{opiod_taper_noncancer} found that longer tapers are typically better. The survey by Fenton et al. \cite{opiod_taper_trends} argues that most tapering protocols which focus primarily on getting through the acute withdrawal phase may be too rapid and have negative mental health consequences. In an observational study, Agnoli et al. \cite{opiod_taper_overdose} find that tapering of opiods is significantly associated with increased risk of overdose and mental health crisis. Henry et al.
 \cite{opiod_taper_focus_group} attempt to characterize patients' subjective experiences with opiod tapering in an effort to minimize negative tapering reactions.

\paragraph{Models of Drug Response and Tolerance.} In the present work, we will build upon a popular model of addiction, proposed by Solomon \cite{opponent_process}, called the Opponent-Process theory of acquired motivation. In this model, the drug response is the result of an initial `positive' effect with short lag and fast decay (A process) followed by a counter `negative' effect with high latency and slow decay (B process). This formalizes the empirical qualitative observation that the withdrawal symptoms of a drug are characterized as opposite to its acute effects \cite{kalant1973biological}. The shape of such a response is plotted in Figure~\ref{fig:opponent}.
\begin{figure}[H]
\centering
\includegraphics[width=0.8\linewidth]{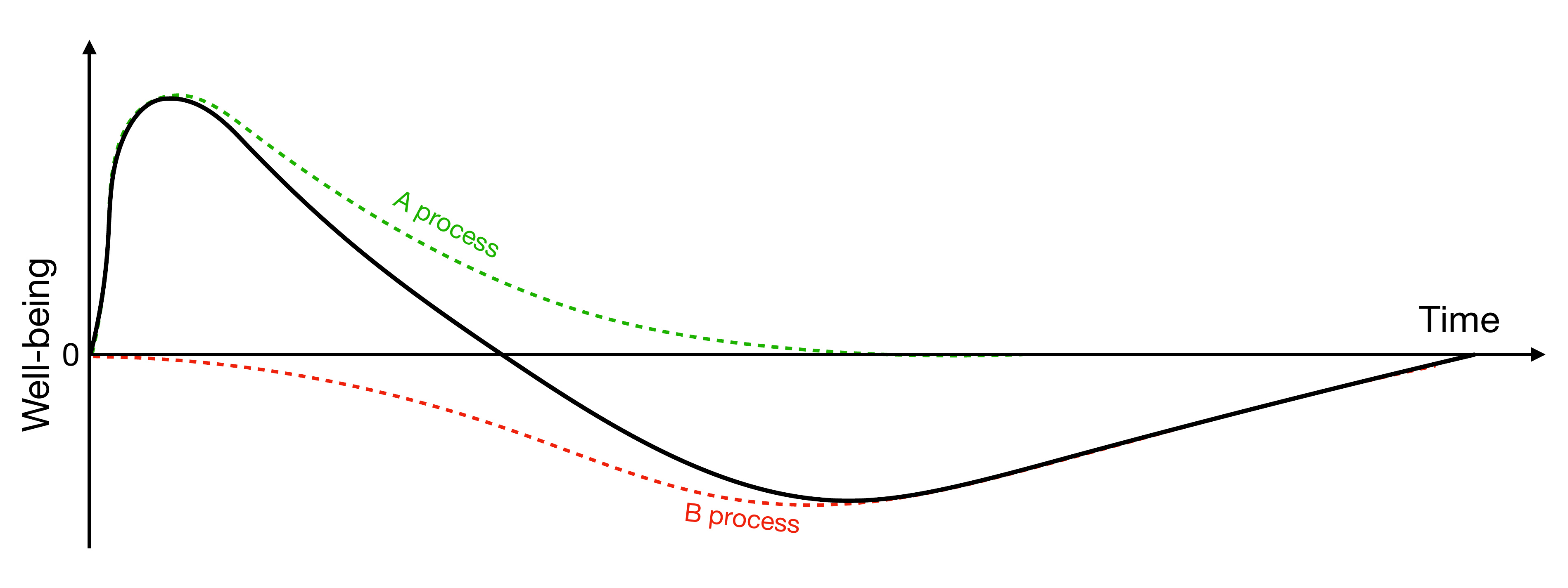}
\caption{Shape of an Instantaneous Opponent Process response as described by Solomon \cite{opponent_process}.}
\label{fig:opponent}
\end{figure}

\noindent Koob and Moal \cite{allostatic} and Koob \cite{koob_addiction} extend the opponent process framework to account for a chronic deviation of the regulatory system from baseline under repeated administration of the drug, which is referred to as allostasis and is illustrated in Figure~\ref{fig:opponent-blunt}. In this work, we propose a simple mathematical formulation of the opponent process that captures the salient aspects of this model. 

\begin{figure}[H]
\includegraphics[width=0.9\linewidth]{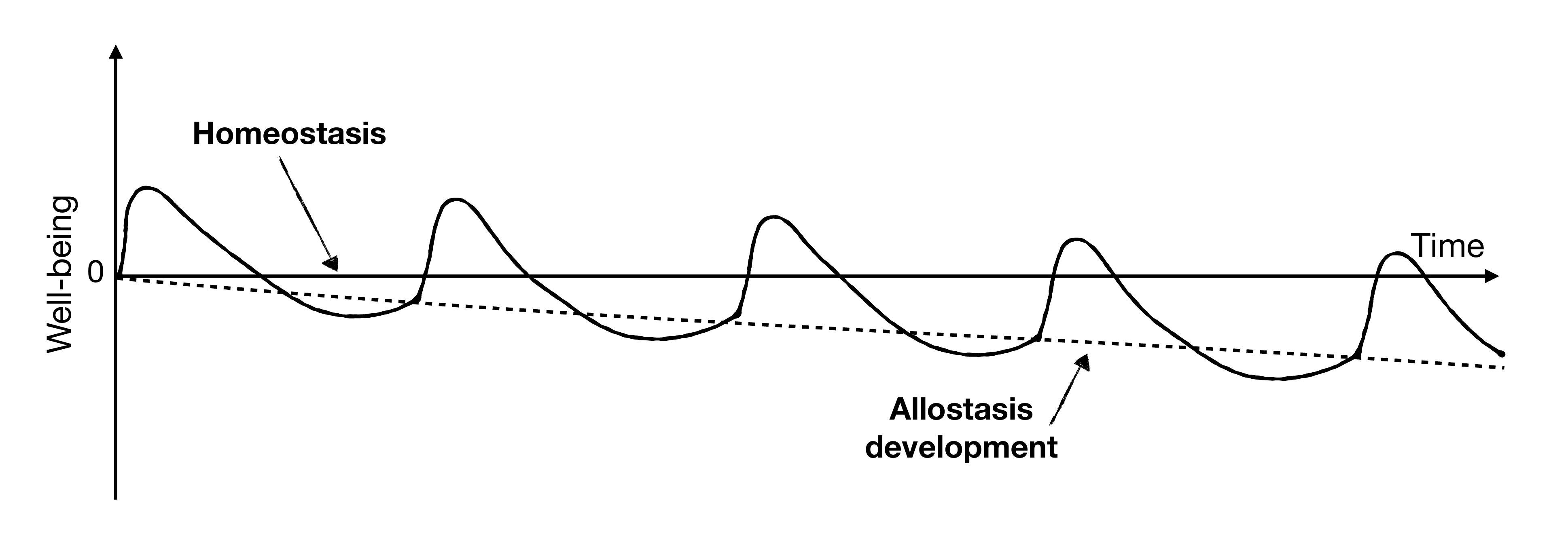}
\caption{Allostasis development due to overlay of opponent processes over time as described by Koob and Moal \cite{allostatic}.}
\label{fig:opponent-blunt}
\end{figure}

\section{Problem Setting and Preliminaries}\label{sec:prelim}

To quantify a substance-agnostic drug response, we will denote an individual's \emph{well-being} at any time $t$ as $y_t$ and the \emph{dose} of a drug taken as $u_t$. This formulation asserts that $y_t$ and $u_t$ are concrete numerical quantities. For $u_t$ this is natural as nearly all prescription/tapering schemes specify dosages in terms of raw active compound weight. For $y_t$, our methods can be instantiated with a wide array of imprecise quantifications of well-being. For example, for tapering anti-depressant medication, $y_t$ could be a weekly score on the Patient Health Questionnare (PHQ-9) \cite{kroenke2002phq}. If we want to instead directly measure withdrawal, we could for example have $y_t$ be the negative of the number of Discontinuation-Emergent Signs and Symptoms (DESS)~\cite{rosenbaum1998selective}. %For managing painkiller medication and/or opiate withdrawal, $y_t$ could be the negative of a measured value on any standard pain assessment scale \cite{haefeli2006pain}, or on the Subjective/Objective Opiate Withdrawal Scales \cite{handelsman1987two}.
\footnote{Since our paper is framed in terms of maintaining `well-being,' we need to take the negative of measurements of withdrawal intensity. The paper's results can all be restated in terms of maintaining withdrawal intensity below a desired individual-specified threshold via a simple sign flip.} 

We assume that the well-being $y_t$, given a history of doses, evolves according to a linear time-invariant (LTI) system
\begin{equation}\label{eq:dynamics}
y_{t+1} = \sum_{k=0}^{t} g(k) u_{t-k} + y^\mathrm{nat}_{t+1}
\end{equation}
where $g$ is the \emph{impulse response} of the drug's effect on well-being, and $y^\mathrm{nat}$ comprises the \emph{natural progression} of the remaining dynamics that cannot be attributed to the recorded drug intake. We focus on linear responses as they are sufficiently general to capture much of the qualitative behavior observed in the psychology literature, and discuss extensions to nonlinear systems in Appendix~\ref{app:nonlin_gen}. Additionally, we allow for arbitrary deviations from the LTI dynamics through $y^\mathrm{nat}$: any measurement noise, exogenous disturbance, and/or effects of doses taken before the time tapering protocol is adopted (set by convention to $t=0$) are implicitly represented through this term. %We will only occasionally lightly model this quantity by asserting lower bounds on its values or bounds on its rate of change, but generally assume it is .

Broadly, the \emph{goal of tapering} is to minimize the consumption of a drug while keeping well-being above some specified threshold $y_\mathrm{min}$. We can directly formalize this as an online optimization problem: over some \emph{tapering horizon} $T$, we aim to minimize the \emph{cumulative dose} $\sum_{t=0}^{T-1} u_t$, while satisfying the \emph{constraint} that $y_t \geq y_\mathrm{min}$ for all $t=\overline{1,T}$.    

\paragraph{Modeling Opponent Processes.}We now turn to giving a mathematical characterization of an opponent processes within our control-theoretic framework. As a reminder, Solomon \cite{opponent_process} asserts that an opponent process reaction to a single dose is first an A-process where $y>0$ and then a following B-process where $y < 0$. We summarize this property in Definition~\ref{def:opp_process} below.

\begin{definition}[Opponent Process]\label{def:opp_process} We say that $g$ is an opponent process if there exists a time $\tau_0$ such that $g(\tau) > 0$ when  $\tau<\tau_0$ and $g(\tau)\leq 0$ if $\tau\geq \tau_0$.
\end{definition}

In this work, we care about opponent processes that both lead to the development of allostasis (i.e., tolerance/addiction) \textit{and} are sufficiently well-behaved so that tapering is possible.  In Definition~\ref{def:well_behaved_opp_process} below, we specify a general class of systems that satisfy these desiderata. 

\begin{definition}[Linearly Progressing Opponent Process]\label{def:well_behaved_opp_process} An opponent process $g$ is a \emph{Linearly Progressing Opponent Process} (LPOP) if there additionally exists an $\alpha \in [0,1)$ such that 
\begin{align*}
    g(t+1) &\leq \alpha \cdot g(t)~~~\,\text{for}~~t <\tau_0 -1,\\
    |g(t+1)| &\geq \alpha \cdot |g(t)|~~\text{for}~~t\geq \tau_0\,.
\end{align*}
\end{definition}

This definition quantifies opponent processes where the A-process decays exponentially, and the B-process does not decay as quickly as the A-process. Qualitatively, 
Definition~\ref{def:well_behaved_opp_process} models a response where the positive effects decay at some rate $\alpha$, then become negative, trending towards some peak of withdrawal, and finally decay back to $0$, potentially at a prolonged rate.

We note that this always happens for opponent processes with $\tau_0=1$ since we can then take $\alpha = 0$. This observation is useful in the sequel and we flag it here specifically as a remark.

\begin{remark}\label{rmk:inst_opp}{\normalfont Any opponent process $g$ with $\tau_0=1$ is an LPOP.}
\end{remark}

\begin{example}{\normalfont Consider a linear system with impulse response
$$
    g(t) = \sum_\lambda c_\lambda \lambda^t
$$
where all of the $\lambda \in [0,1)$ and the $c_\lambda$ are real valued scalars. The terms where $c_\lambda$ are positive correspond to effects that increase well-being. The terms where $c_\lambda$ are negative decrease well-being. With this in mind, define $\Lambda_{+} = \{\lambda ~:~ c_\lambda \geq 0\}$ and $\Lambda_{-} = \{\lambda ~:~ c_\lambda < 0\}$. Then, the system is an LPOP if the following three conditions hold:
\begin{enumerate}
\item    $\max_{\lambda \in \Lambda_{+}} \lambda  \leq \min_{\lambda' \in \Lambda_{-}} \lambda'$
\item   $\sum_\lambda c_\lambda>0$
\item There exists some $\tau_0$ s.t. $g(\tau_0) \leq 0$
\end{enumerate}}
\end{example}

Intuitively, this holds because the part of the system corresponding to positive effects decay more rapidly than all of the terms corresponding to negative effects. In particular, the second condition implies an initial positive response, and third condition implies that after a single dose, the well-being will eventually be negative.

The formal argument of why such systems are LPOPs proceeds by bounding $g(t+1)$ in terms of $g(t)$. If we let $\Lambda_{++} := \max_{\lambda \in \Lambda_{+}} \lambda$ and $\Lambda_{--}:= \min_{\lambda' \in \Lambda_{-}} \lambda'$, then we have

\begin{align*}
g(t+1) = \sum_\lambda c_\lambda \lambda^{t+1} &= \sum_{\lambda_+ \in \Lambda_+} c_{\lambda_+} \lambda_+^{t+1} + \sum_{\lambda_- \in \Lambda_-} c_{\lambda_-} \lambda_-^{t+1} \\
&\leq \Lambda_{++} \cdot \sum_{\lambda_+ \in \Lambda_+} c_{\lambda_+} \lambda_+^{t} + \Lambda_{--} \cdot \sum_{\lambda_- \in \Lambda_-} c_{\lambda_-} \lambda_-^{t} \\
&\leq \Lambda_{--} \left(\sum_{\lambda_+ \in \Lambda_+} c_{\lambda_+} \lambda_+^{t} + \sum_{\lambda_- \in \Lambda_-} c_{\lambda_-} \lambda_-^{t} \right) \\
&=  {\Lambda_{--}} g(t)
\end{align*}
Setting $\alpha\doteq \Lambda_{--}$, the above implies that for $t < \tau_0 - 1$, $g(t+1) \leq \alpha g(t)$ and for $t \geq \tau_0$, $|g(t+1)| \geq \alpha |g(t)|$. Hence $g$ satisfies Definition~\ref{def:well_behaved_opp_process}. \\

In Figure~\ref{fig:drug_behavior} we plot the dose response and the effect deterioration observed when taking a fixed dose for four examples of 2-pole LTI systems, each capturing an LPOP with different time constants and relative effect magnitudes. 

%We invite the reader to play with different instantiations using the code provided \href{https://github.com/paula-gradu/tapering}{here}.

\begin{figure}[H]
\centering
\includegraphics[width=\linewidth]{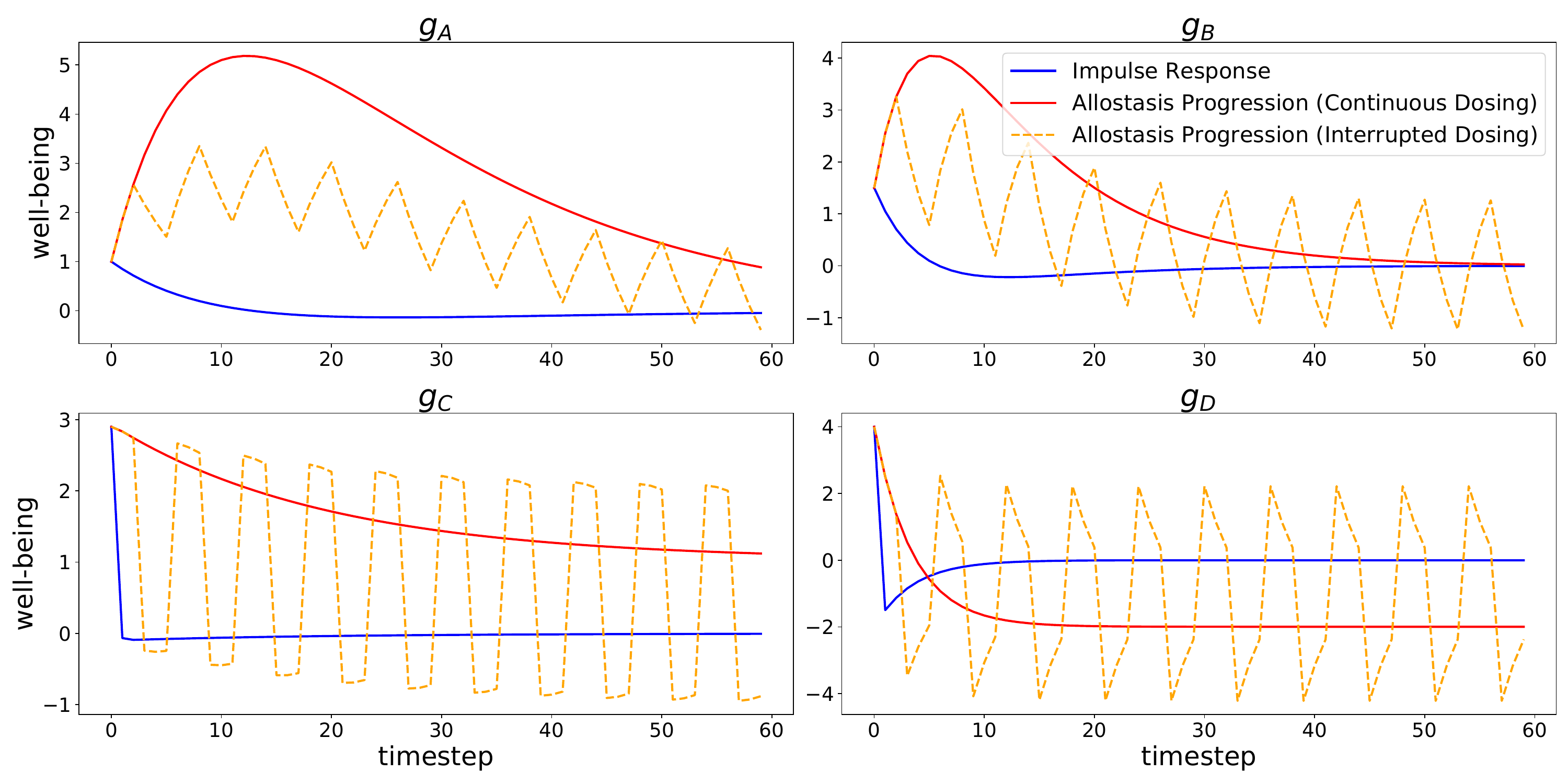}
\caption{Impulse responses and allostasis progression over $60$ timesteps for four opponent processes. The top panels illustrate opponent processes with slowly accumulating benefits and tolerance, characteristic of many therapeutic medications. The bottom panels capture opponent processes with immediate transient effects, and either subtle yet slowly accumulating negative effects (left) or a rapid onset of very strong negative effects (right), the latter being more characteristic of substances of abuse. The well-being is measured in arbitrary units which could be scaled to fit a particular metric of interest.}
\label{fig:drug_behavior}
\end{figure}

\section{Optimal Tapering With Full System Knowledge}\label{sec:opt-result}

In this section, we derive the optimal tapering protocol when \textit{all} the quantities described in the previous section are known. In particular, we show that the optimal solution to the tapering control problem formalized above is achieved by taking the dose that induces exact constraint matching, which is equivalent to setpoint matching with setpoint $y_\mathrm{min}$. While not directly actionable, this insight provides a path forward towards designing new adaptive tapering protocols with stronger individual guarantees, as we will see in Section~\ref{sec:integral}.

We begin with the formal definition of the \emph{minimum effective dose}, the smallest dose that does not lead to an immediate violation of the constraint on well-being. The main theorem of this section states that taking the minimum effective dose at every step ensures perfect constraint satisfaction with minimal cumulative dose over the entire tapering horizon . %This reveals that, for well-behaved opponent-processes, optimal tapering is actually an instance of setpoint matching in disguise.

\begin{definition}[Minimum Effective Dose]\label{def:greedy_dose} We say that a dose $u_t$ is the \emph{minimum effective dose (MED)} at time $t$ if it either leads to $y_{t+1} = y_\mathrm{min}$ or is equal to $0$.
\end{definition}

\begin{theorem}\label{thm:max_greedy_opt} For any LPOP $g$ and any natural progression $y^\mathrm{nat}$, taking the MED at every time ensures $y_t \geq y_\mathrm{min}$ for all $t$ with minimal cumulative dose. \end{theorem}

\begin{proof} Consider a dosing sequence $\{u^\star_t\}_{t=0}^{T-1}$ that maintains $y_t^\star \geq y_\mathrm{min}$ for all $t$ with minimal cumulative dose. We will prove the theorem by contradiction: assume there exists a time $t_{0}$ when $u_{t_0}^\star$ is not the MED. Then, there must exist $\epsilon > 0$ such that replacing $u_{t_{0}}^\star$ with $u_{t_{0}}^\star - \epsilon$ would result in a slightly lower well-being measurement $y_{t_0+1}' < y_{t_0+1}^\star$ that still satisfies the constraint $y_{t_0+1}' \geq y_\mathrm{min}$. We can then consider the following modified dosing schedule $\{u_t'\}_{t=0}^{T-1}$:
$$u_t'= \begin{cases} u_{t_0}^\star - \epsilon & t=t_0 \\ u_{t_0+1}^\star + \alpha \epsilon & t=t_0+1 \\ u_t^\star & \text{ otherwise}\end{cases}$$
where $\alpha$ is the LPOP's separating decay rate constant from Definition~\ref{def:well_behaved_opp_process}. Let $\{y_t'\}_{t=1}^T$ be the well-being sequence resulting from this modified dosing schedule. 

Since $u_t' = u_t^\star$ for all $t < t_0$, we have that $y_t' = y_t^\star \geq y_\mathrm{min}$ for all $t \leq t_0$. Plugging in the effect of the dose modifications at $t_0$ and $t_0 + 1$ into the dynamics equation \eqref{eq:dynamics}, we get:
\begin{align*}
y_{t+1}' &= y_{t+1}^\star - g(t-t_0) \epsilon + g(t-t_0-1) \alpha \epsilon \\
&\geq y_{t+1}^\star \\
&\geq y_\mathrm{min}
\end{align*}
where the first inequality follows by the definition of an LPOP (Definition~\ref{def:well_behaved_opp_process}). Since this holds for all $t > t_0$, we have that $y_t'\geq y_\mathrm{min}$ for all $t > t_0 + 1$ as well.

Therefore the well-being sequence $\{y_t'\}_{t=1}^T$ also satisfies the well-being constraints for all $t$. However, since $\alpha < 1$, the modified dosing schedule $\{u_t'\}_{t=0}^{T-1}$ that lead to it has strictly smaller cumulative dose than $\{u_t^\star\}_{t=0}^{T-1}$:
$$\sum_{t=0}^{T-1} u_t'=\sum_{t=0}^{T-1} u_t^\star - \epsilon + \alpha \epsilon < \sum_{t=0}^{T-1} u_t^\star\,.$$
This is a contradiction, completing the proof.
\end{proof}

This theorem shows that long-term planning is not necessary to optimally taper a medication. It always suffices to take the smallest dose that ensures satisfaction of just the very next well-being constraint. Given this insight, the main question becomes how to compute the MED.

Let us first assume that we know the full specification of the impulse response $g$, and that we are given (or can iteratively construct) a sequence $\{y_t^\mathrm{nat, lb}\}_{t=1}^T$ that lower bounds all potential underlying natural progressions $y^\mathrm{nat}$, i.e. 
\begin{equation}\label{eq:lb_admissible}
y^\mathrm{nat}_t \geq y^\mathrm{nat, lb}_t \text{ for all } t.\tag{*}  
\end{equation}
With this information we can directly derive the MED, and therefore an explicit-form protocol given below in Proposition~\ref{thm:opt_ynat}.
\begin{proposition}\label{thm:opt_ynat} For any LPOP $g$, taking the dose
\begin{equation}\label{eq:opt_ynat}
u_t = \max\left\{0, \frac{y_\mathrm{min} - y_{t+1}^\mathrm{nat, lb} - \sum_{k=1}^t g(k) u_{t-k}}{g(0)}\right\}
\end{equation} 
ensures $y_t\geq y_\mathrm{min}$ for all $t$ on all natural progressions satisfying ($\star$) with minimal cumulative dose.
\end{proposition}

\begin{proof} First we show that the doses $u_t$ prescribed by (\ref{eq:opt_ynat}) maintain $y_t\geq y_\mathrm{min}$ for all $t$ under any natural progression satisfying ($\star$). To do so, we start by deriving the MED at an arbitrary time $t_0$ for a particular instantiation $\{\tilde{y}_t^\mathrm{nat}\}_{t=1}^T$. Given any past doses $\{u_t\}_{t=0}^{t_0-1}$, the MED is equal to
\begin{equation}\label{eq:opt}
\tilde{u}_{t_0}^\mathrm{MED} = \max\left\{0, \frac{y_\mathrm{min} - \tilde{y}_{t_0+1}^\mathrm{nat} - \sum_{k=1}^{t_0} g(k) u_{t_0-k}}{g(0)}\right\}
\end{equation}
To see why, let $\tilde{y}_{t_0+1}$ denote the resulting well-being, and observe that the dose is either $0$ or satisfies
$$g(0) \, \tilde{u}_{t_0}^\mathrm{MED} + \sum_{k=1}^{t_0} g(k) \, u_{t_0-k} + \tilde{y}_{t_0+1}^\mathrm{nat} = y_\mathrm{min} $$

By the dynamics equation \eqref{eq:dynamics}, the left hand side is precisely equal to $\tilde{y}_{t_0+1}$. This means that if $\tilde{u}_{t_0}^\mathrm{MED}$ is not equal to $0$ then it results in a next iterate well-being measurement that exactly matches the constraint, which is precisely the definition of the MED. 

Because, having fixed past doses, \eqref{eq:opt} is solely determined by $\tilde{y}_{t_0+1}^\mathrm{nat}$ and is decreasing with respect to it, we have that $u_{t_0} \geq \tilde{u}_{t_0}^\mathrm{MED}$ and therefore $y_{t_0+1} \geq \tilde{y}_{t_0+1}\geq y_\mathrm{min}$ for any $\tilde{y}_{t_0}^\mathrm{nat} \geq y_{t_0}^\mathrm{nat,lb}$. Since this holds independently for all $t_0$ (irrespective of past doses), we have that \eqref{eq:opt_ynat} maintains $y_t\geq y_\mathrm{min}$ for all $t$ and any natural progression satisfying ($\star$), as promised.

To see that it achieves the above with minimal cumulative dose, observe that at every time $t$ the dose $u_t$ prescribed is \textit{exactly} the MED corresponding to setting $\tilde{y}_{t+1}^\mathrm{nat}$ to its lowest instantiation ($y_{t+1}^\mathrm{nat, lb}$) in \eqref{eq:opt}. Therefore, Theorem~\ref{thm:max_greedy_opt} implies that any dosing sequence with a strictly lower cumulative dose will eventually lead to $y_t < y_\mathrm{min}$ on the natural progression which exactly matches the lower bound in ($\star$). This means that the proposed protocol \eqref{eq:opt_ynat} ensures $y_t\geq y_\mathrm{min}$ for all $t$ on all natural progressions satisfying ($\star$) with minimal cumulative dose, concluding the proof. 
\end{proof}

There are two natural approaches to lower bound $y^\mathrm{nat}_{t+1}$ at time $t$ in order to compute $u_t$ prescribed by \eqref{eq:opt_ynat}. First, we may assume that $y^\mathrm{nat}_{t+1} \geq y^\mathrm{nat}_t$ for all $t$. This would correspond to assuming that the underlying natural progression is monotonically non-decreasing. In this case, the best lower bound candidate is $y^\mathrm{nat, lb}_{t+1} = y^\mathrm{nat}_{t}$ which we can compute via 
\begin{equation}\label{eq:ynat_form}
y^\mathrm{nat}_t = y_t - \sum_{k=0}^{t-1} g(k) u_{t-k-1}\,.
\end{equation}
Second, we could assume that there exists a constant $L_\mathrm{nat} \geq 0$ such that $y^\mathrm{nat}_{t+1} \geq y^\mathrm{nat}_t - L_\mathrm{nat}$ for all $t$. This would correspond to assuming that the underlying natural progression cannot trend downwards in a single step at a rate faster than $L_\mathrm{nat}$. In this case, the best lower bound candidate is $y^\mathrm{nat, lb}_{t+1} = y^\mathrm{nat}_{t} - L_\mathrm{nat}$, where we again compute $y^\mathrm{nat}_{t}$ via \eqref{eq:ynat_form}.\\

Finally, one may wonder if the above principle for optimal tapering extends to more general settings. A simple yet important scenario is when the constraints are time-varying. This is relevant when a person can tolerate more intense withdrawal effects at particular times of their lives or when they can no longer tolerate an initially agreed-upon baseline. Remark~\ref{rmk:opt_tv} shows that optimal tapering is still possible with a time-varying lower bound and becomes equivalent to constraint tracking in this case.

\begin{remark}\label{rmk:opt_tv}{\normalfont The proof of Theorem~\ref{thm:max_greedy_opt} is valid even if the well-being constraints are time-varying. Therefore, optimal tapering is again achieved by taking the MED at every time $t$ with respect to the very next well-being constraint $y_\mathrm{min}^{(t+1)}$. Furthermore, the protocol from Proposition~\ref{thm:opt_ynat} can be directly extended to this setting by replacing $y_\mathrm{min}$ in \eqref{eq:opt_ynat} with its time-varying counterpart $y_\mathrm{min}^{(t)}$.}
\end{remark}

The above remark is still true even if the dose responses $g$ are time-varying. Beyond being able to encompass non-linear affine systems, time-varying opponent processes cover situations when there is some process other than dependency-induced allostasis altering the dose-response over time. Going even further, in Appendix~\ref{app:nonlin_gen} we consider non-linear opponent processes and show that the main results of this section still hold under an appropriate generalization of Definition~\ref{def:well_behaved_opp_process}. Finally, note that by Remark~\ref{rmk:inst_opp} (and its extension to non-linear systems in Appendix~\ref{app:nonlin_gen}), \textit{any} opponent process discretized to have $\tau_0 = 1$ satisfies the well-behavedness assumption and therefore can be optimally tapered by taking the MED at every time. The MED is optimal under surprising generality.

\section{Tapering By Integral Control}\label{sec:integral}

We saw in the previous section that optimal tapering is achieved by always taking the smallest dose that satisfies the next well-being constraint. However, computing this dose required knowledge of the model of the underlying dynamics $g$, as well as a way to bound the next iterate in the underlying natural progression $y^\mathrm{nat}_{t+1}$. In practice, we would like to avoid estimating these two quantities and instead derive an update rule requiring minimal knowledge.

Since we know that optimal tapering is achieved by constraint matching (up to dose positivity) and that integral control can match setpoints asymptotically, in this section we propose and analyze a simple integral controller that takes takes the current well-being $y_t$ as input and uses its deviation from the constraint $y_\mathrm{min}$ to iteratively compute dose corrections.

We demonstrate that this simple approach can reduce a dose while rarely violating the constraints when the gains are chosen appropriately. In fact, one needs to only know a rough range of the immediate effect of a single dose, which is equal to the value $g(0)$. Moreover, when the dynamics are discretized coarsely enough, we show that the resulting doses monotonically decrease to zero in finite time while ensuring \textit{perfect} constraint satisfaction (i.e. $y_t \geq y_\mathrm{min}$ for all $t$).

To proceed, let $(x)_+:=\max(x,0)$ and $(x)_-:=\min(x,0)$
Consider the integral control law
\begin{equation}\label{eq:simple_update_true}
u_t =  \max \biggl\{0,u_{t-1} - K_+ (y_t-y_\mathrm{min})_+ - K_-(y_t-y_\mathrm{min})_- \biggl\}\,.
\end{equation}
There are two differences between the above protocol and a canonical integral controller. First, the control action is clipped at $0$ as it is impossible to take a negative dose of a medication. Second, it allows for two gains: $K_+$ for when $y_t$ is above the desired ``setpoint'' $y_\mathrm{min}$, and $K_-$ for when it falls below. Using two gains allows for more conservative dose corrections when $K_+ < K_-$, leading to smaller decrements when $y_t > y_\mathrm{min}$ and bigger increments when $y_t < y_\mathrm{min}$. 

The main theorem of this section bounds the \emph{long-term constraint violation} \cite{long_term_constraints, mannor2006online} of the proposed integral controller. %We aim to study how the average value of $y_t$ compares to the constraint $y_{\mathrm{min}}$ over time. 
The bound we derive ensures that the running average of the well-being observations $y_t$ up until a time $T$ is greater than $y_\mathrm{min}$ minus a penalty decaying at rate $\frac{1}{T}$. This means that the overall well-being satisfies the constraint more strictly as time increases. 

\begin{theorem}\label{thm:general_cumulative_viol} For any underlying $y^\mathrm{nat}$, doses prescribed by (\ref{eq:simple_update_true}) ensure that for any $T$
$$
\frac{\sum_{t=1}^T y_t}{T} \geq y_\mathrm{min} -  \frac{y_0 - y_{\mathrm{min}}}{T}\,,
$$
provided $K_+\leq g(0)^{-1} \leq K_-$.
\end{theorem}

The proof of this theorem is broadly inspired from analyses of the finite time behavior of average constraint violation in regret minimization by
Mannor and Tsitsiklis~\cite{mannor2006online} and
Mahdavi et al.~\cite{long_term_constraints}. To prove the theorem, we need the following technical lemma that lower bounds the instance-wise deviation of the $y_t$ obtained from the dosing scheme (\ref{eq:simple_update_true}) from $y_\mathrm{min}$ if the gains $K_+$/$K_-$ upper/lower bound $g(0)^{-1}$ respectively. The proof of the lemma is given in Appendix~\ref{app:lem_proof}.

\begin{lemma}\label{lem:inst_viol} For any underlying $y^\mathrm{nat}$, and any gains satisfying $K_+\leq g(0)^{-1} \leq K_-$, setting $u_t$ according to (\ref{eq:simple_update_true}) for $t\geq 0$ yields $y_{t+1}$ satisfying
$$y_{t+1} \geq y_\mathrm{min} + (y^\mathrm{nat}_{t+1} - y^\mathrm{nat}_t) - \sum_{k=1}^{t} g(k) (u_{t-k-1} - u_{t-k})\,.$$
where we adopt the convention that $u_{-1}=0$.
\end{lemma}

With this lemma, we can now prove Theorem~\ref{thm:general_cumulative_viol}.

\begin{proof}[Proof of Theorem~\ref{thm:general_cumulative_viol}]  Using Lemma~\ref{lem:inst_viol}, we have that
\begin{align*}
\sum_{t=0}^{T} y_{t+1} &\geq (T+1) y_\mathrm{min} + \sum_{t=0}^{T} \big((y^\mathrm{nat}_{t+1} - y^\mathrm{nat}_t) - \sum_{k=1}^{t} g(k) (u_{t-k-1} - u_{t-k})\big) \\
&= (T+1) y_\mathrm{min} - y^\mathrm{nat}_0 + y^\mathrm{nat}_{T+1} + \sum_{t=0}^{T} u_t g(T-t) \\
&= (T+1) y_\mathrm{min} + y_{T+1} - y_0 
\end{align*}
where the first equality follows by telescoping, and the second from the implicit convention in \eqref{eq:dynamics} that $y_0^\mathrm{nat} = y_0$. By cancelling $y_{T+1}$ on both sides, we get $$\sum_{t=0}^{T-1} y_{t+1} \geq T y_\mathrm{min} - (y_0 - y_\mathrm{min})$$
which can be immediately rearranged into the stated bound:
\begin{align*}
\frac{\sum_{t=1}^{T} y_{t}}{T} = \frac{\sum_{t=0}^{T-1} y_{t+1}}{T} \geq y_\mathrm{min} - \frac{y_0 - y_\mathrm{min}}{T} 
\end{align*}
\end{proof}

\begin{remark}\label{rmk:padding}{\normalfont If it is undesirable to fluctuate too much below $y_\mathrm{min}$, we can run a padded integral controller by adding $\delta$ to $y_\mathrm{min}$. Theorem~\ref{thm:general_cumulative_viol} then implies 
$$\frac{\sum_{t=1}^T y_t}{T} \geq y_\mathrm{min} + \delta - \frac{y_0 - y_{\mathrm{min}} -\delta}{T}\,.$$}
\end{remark}
\begin{remark}\label{rmk:dev_bd_tv} {\normalfont
As in the previous section, we can extend the integral controller (\ref{eq:simple_update_true}) to time-varying constraints by plugging in $y_\mathrm{min}^{(t)}$. In this case, the guarantee from Theorem~\ref{thm:general_cumulative_viol} would instead bound deviation from the average $y_\mathrm{min}$, i.e. from $T^{-1} \sum_{t=1}^T y_\mathrm{min}^{(t)}$.}    
\end{remark}

We end the technical component of this section by deriving stronger guarantees for opponent processes with $\tau_0=1$, a subset of the LPOP class considered thus far. Requiring $\tau_0 = 1$ can be interpreted as an assumption on the discretization of the dose updating scheme (e.g., every day vs. every week). For example, we can turn a system $g$ with $\tau_0 > 1$ into a system $g'$ with $\tau_0' = 1$ by taking $g'(t) = \tau_0^{-1} \cdot \sum_{t'=t\cdot \tau_0}^{(t+1) \cdot \tau_0 - 1}g(t')$.

Our final result shows that for such suitably discretized opponent processes, when the underlying natural progression does not trend down too fast, we can guarantee that the proposed integral controller ensures \textit{perfect} constraint satisfaction, as well as monotonically non-increasing doses that reach $0$ in finite time. 

\begin{proposition}\label{prop:tau_0_eq_1}
For any $g$ with $\tau_0= 1$, any initial $u_0$ which ensures $y_1\geq y_\mathrm{min}$, any $\delta>0$, and any natural progression sequence satisfying $y^\mathrm{nat}_{t+1} \geq y^\mathrm{nat}_t - g(t)u_0 + \delta/t$, doses $u_t$ prescribed by (\ref{eq:simple_update_true}): (i) are non-increasing, (ii) maintain $y_t\geq y_\mathrm{min}$ for all $t$, and (iii) are equal to zero after some finite time $T_0$.
\end{proposition}

\begin{proof} To prove (i) and (ii) jointly, we will proceed by induction on $t$. The base case holds by the assumption that we have access to an initial $u_0$ for which $y_1\geq y_\mathrm{min}$. Assume for the inductive hypothesis that $0 \leq u_{t-1} \leq \ldots \leq u_0$ and $y_t\geq y_\mathrm{min}$. First, observe that since $y_t - y_\mathrm{min} \geq 0$, we have
$$u_t \leq \max\{0, u_{t-1}\} = u_{t-1}$$
showing the inductive step for (i). To show the inductive step for (ii), note that by Lemma~\ref{lem:inst_viol},
\begin{align*}
y_{t+1} &\geq y_\mathrm{min} + (y^\mathrm{nat}_{t+1} - y^\mathrm{nat}_t) - \sum_{k=1}^t g(k) (u_{t-k-1} - u_{t-k})  \\
&\geq y_\mathrm{min} + (y^\mathrm{nat}_{t+1} - y^\mathrm{nat}_t) + g(t)u_0 \\
&\geq y_\mathrm{min} + \delta/t 
\end{align*}
where the second inequality follows since $g(k) < 0$ and $u_{t-k-1} \geq u_{t-k}$ for $k=\overline{1,t-1}$, while the last inequality follows from our assumption that $y^\mathrm{nat}_{t+1} \geq y^\mathrm{nat}_t - g(t) u_0 + \delta/t$.

Therefore, by induction, (i) and (ii) are true. Finally, by the calculation above, we know that $u_t$ is decreasing and, further, that it can be expressed as $$u_t = \max\biggl\{0, u_0 - K_+ \delta\cdot \sum_{k=1}^t \frac{1}{t}\biggl\}$$
which, since $\sum_{k=1}^t 1/t \geq \ln{t}$, implies $u_t = 0$ for all $t\geq \exp\biggl\{\dfrac{u_0}{K_+ \delta}\biggl\}$, proving (iii).
\end{proof}

Finally, What would actually be needed to implement this protocol? A patient well versed in subjective symptom scoring would give a minimally tolerable value for the day. The would compare this value to how well they felt yesterday. Based on this difference, their new dose would simply be given by Equation~\eqref{eq:simple_update_true}. Such a scheme could be easily implemented in a spreadsheet.

As we will show in the experiments below, the protocol is not sensitive to the choice of $K_\pm$. A rule of thumb value for setting $K_\pm$ could be devised based on the smallest dose change that yields a reasonably perceptible change in $y$. So, for instance, if the smallest noticeable dosage change would be $5\text{mg}$ and it would lead to a change in $y$ by somewhere between $1$ and $2$ `units', then $K_+$ should be roughly $2.5$ and $K_-$ should be $5$. 

\section{Numerical Simulations}\label{sec:experiments}

In this section, we conduct numerical simulations to gain insights into the performance of the optimal and the integral-based protocols compared to that of standard (non-adaptive) tapering approaches when deployed on a population with different well-being constraints. Due to a lack of available simulators and open source data for our problem, we rely on synthetic impulse response functions, modelled to induce realistic and varied impulse responses and allostatic adaptations. We describe the experiment setup in detail below, and plot the results in Figure~\ref{fig:tapering_results}.

\paragraph{Task specification.}We study the problem of tapering in the four opponent processes from Figure~\ref{fig:drug_behavior}, starting from the allostatic state reached after taking dose $u=1$ for $T_\mathrm{init}=60$ timesteps. For each setting, we consider tapering a population of $N$ units with uniformly distributed constraints $y_\mathrm{min}$, over some time frame $T_\mathrm{taper}$. %(Note that it is infeasible to request $y_\mathrm{min}$ greater than the allostatic equilibrium).
Concretely, we take $$y_\mathrm{min}^A \sim \text{Unif}(-1.5, 0.5), \; y_\mathrm{min}^B \sim \text{Unif}(-2, 0), \; y_\mathrm{min}^C \sim \text{Unif}(-1, 1), \; y_\mathrm{min}^D \sim \text{Unif}(-4.25, -2.25)$$ $$\text{and }\; T_\mathrm{taper}^A = 180, \; T_\mathrm{taper}^B = 120, \; T_\mathrm{taper}^C = 90, \; T_\mathrm{taper}^D = 15.$$ For a given individual and a given tapering protocol, we measure the average cumulative dose taken and the average cumulative constraint violation\footnote{Note that this differs from the long-term constraint violation analyzed in Theorem~\ref{thm:general_cumulative_viol}. Here we measure the average over constraint violations $(y_\mathrm{min} - y)_+$, rather than the difference between the average well-being and $y_\mathrm{min}$.} over the tapering horizon $T_\mathrm{taper}$. We then average these two metrics over the $N$ units to obtain the average performance of a protocol over a population. Finally, for all setups, we also perturb $y$ at every time with $\text{Unif}(-0.25, 0.25)$ noise in order to capture random deviations in an individual's day-to-day well-being. 

\paragraph{Specification of tapering protocols.}As baselines, we consider the linear ($u_t = u_0 -\alpha t$) and the exponential ($u_t= \alpha^t u_0$) dose decay protocols. To evaluate the protocol designs rather than a particular substance-dependent instantiation, we compute the population performance in terms of the above two metrics over a range of decay rates $\alpha$. Doing so also gives us a comprehensive trade-off curve between constraint violation and dosage reduction for different rates.

To evaluate the integral controller~\eqref{eq:simple_update_true}, we need to specify the patient constraint $y_\mathrm{min}$ and the gains $K_{-/+}$. The constraint can directly be provided by the patient so we feed it directly to the protocol. Since our method can be instantiated with a padded $y_\mathrm{min}^\delta = y_\mathrm{min} + \delta$ to provide stronger constraint violation guarantees (Remark~\ref{rmk:padding}), we create a trade-off curve by sweeping over different settings of $\delta$. To select the gains, we assume access to a course range of the instantaneous response $g(0)$. In the experiment below we specify particular settings for the sake of clarity, and provide an ablation over a multitude gain specifications in Appendix~\ref{app:exps} instead. Concretely, we take the lower and upper bounds to be $50\%$ and $150\%$ of the true response, which corresponds to the following conservative gain settings: 
$$K_+ = (2/3) \cdot g(0)^{-1} \text{ and } K_-=2 \cdot g(0)^{-1}.$$
Finally, we compare to the optimal tapering protocol, for which we use the full model $g$ and the true underlying natural progression to compute the MED via \eqref{eq:opt}.

\begin{figure}
\centering
\includegraphics[width=\linewidth]{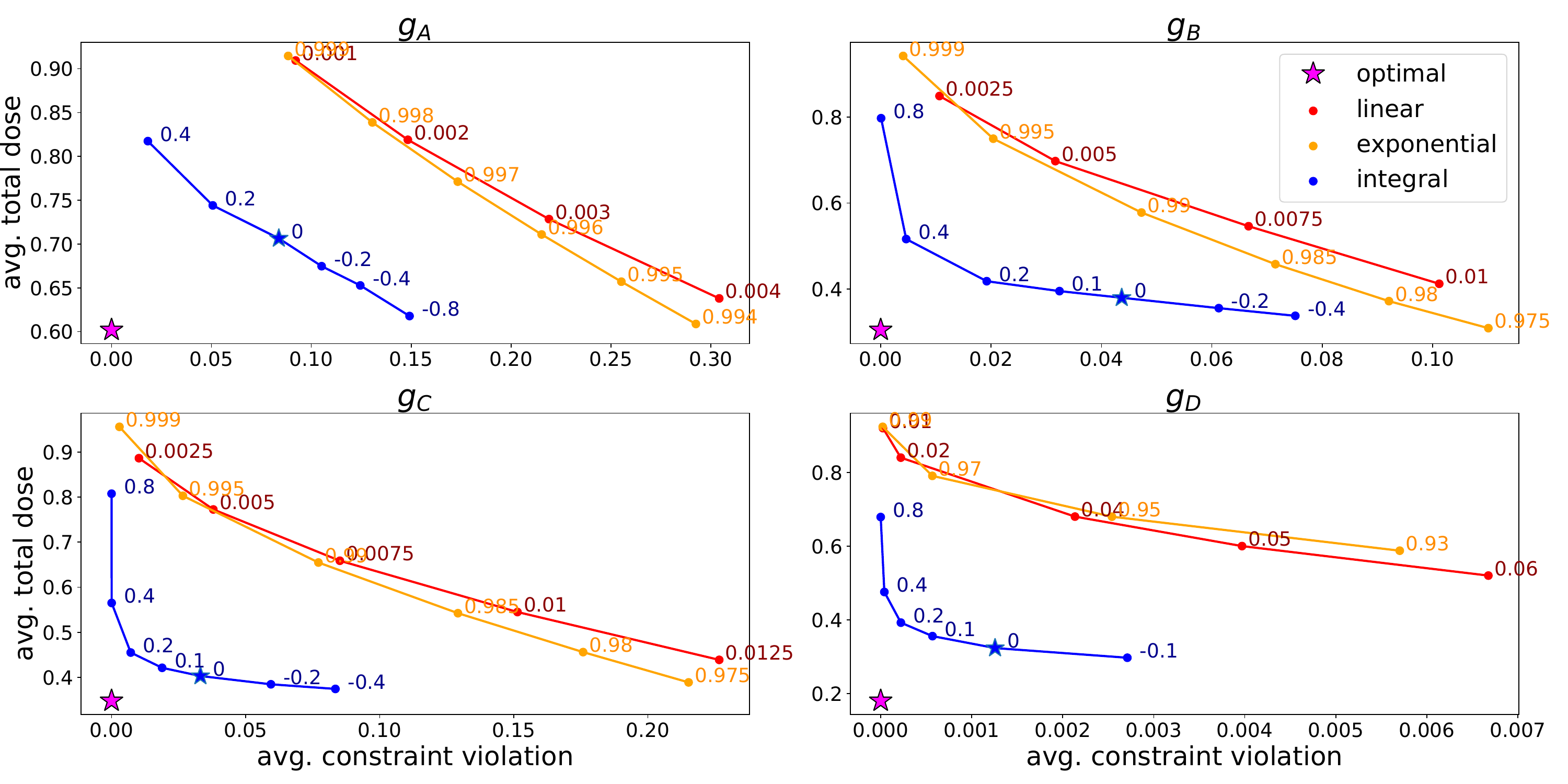}
\caption{Average cumulative constraint violation (x-axis) against average cumulative dose (y-axis). Each point is averaged over $N=100$ units.}
\label{fig:tapering_results}
\end{figure}

\paragraph{Results.}In Figure~\ref{fig:tapering_results}, we see that the integral control protocol dominates the baselines considered, as the curves obtained by sweeping over the various parameters do not intersect. Further, the above results corroborate the prior findings discussed in Section~\ref{sec:related} that very slow rates are necessary to maintain acceptable levels of withdrawal management over a population when restricted to the existing non-adaptive approaches. Finally, there is a non-negligible performance gap between the proposed integral tapering protocol and the optimal protocol (achieved by MED dosing) which opens an avenue for future tapering research.

In the Appendix, we provide further experiments covering the robustness of the integral controller to even more extreme misspecification of the immediate response $g(0)$, as well as performance in terms of an additional metric of potential interest: the percentage of units that successfully taper within the specified time frame. We also invite the reader to explore different settings for the systems and/or protocols by accessing our code \href{https://github.com/paula-gradu/tapering}{here}.

\section{Conclusion \& Discussion}\label{sec:discussion}

There remain several control-theoretic considerations in tapering protocol design. The desirable guarantee of monotonic dose decrease was induced by a discretization such that the full positive effect occurs within one timestep. How can we find the appropriate time window without excessive exploration? Can we develop methods for estimating the switching point of the process and discretizing such that the induced discrete system has $\tau_0 = 1$? 

Additionally, numerous practical, clinical, and ethical considerations must be considered for the downstream deployment of the proposed methods. For example, neither the optimal policy nor our default integral control policy decrease the dose monotonically. For some individuals, this could result in recommending increasing dosages, which may be undesirable in clinical settings. A translation to practice may necessarily add the constraint that a dose never exceeds some level. In turn, this may mean that desired levels of well-being are not achievable for some patients. 

To our knowledge, this work is the first formalization of substance tapering as an optimal control problem. This formalism leads to a concrete approach that departs from existing literature on tapering and prescribes adaptivity. Excitingly, the proposed protocols may be able to address some of the limitations of current approaches while still providing simple, explainable rules. We hope that the straightforward protocols are amenable to undergoing human subject validation in future studies.

\bibliographystyle{plain}
\bibliography{ref}

\newpage
\appendix

\section{Proof of Lemma~\ref{lem:inst_viol}}\label{app:lem_proof}

We first derive an alternate expression of the (possibly negative) dose $u_t^\star$ that leads to $y_\mathrm{min}$. Concretely, using the same derivation in Propositon~\ref{thm:opt_ynat} of the maximally greedy protocol (\ref{eq:opt}), we have that
\begin{equation*}
u_t^\star =\frac{y_\mathrm{min} - y_{t+1}^\mathrm{nat} - \sum_{k=1}^t g(k) u_{t-k}}{g(0)}   
\end{equation*}
We aim to express this in terms of the current set-point error $y_t - y_\mathrm{min}$ and $u_{t-1}$, so we add and subtract $y^\mathrm{nat}_t$ and $\sum_{k=1}^{t-1} u_{t-k-1}$ (which generate $y_t - g(0) u_{t-1}$) to obtain:
\begin{align*}
u_t^\star &= g(0)^{-1}\big(y_\mathrm{min} - y^\mathrm{nat}_t + (y^\mathrm{nat}_t - y^\mathrm{nat}_{t+1}) \\
&\quad - \sum_{k=1}^{t-1} g(k) u_{t-k-1} - \sum_{k=1}^{t} g(k) (u_{t-k} - u_{t-k-1})\big)
\end{align*}
\noindent where we take $u_{-1} =0$. Since $y^\mathrm{nat}_t + \sum_{k=1}^{t-1} g(k) u_{t-k-1} = y_t - g(0) u_{t-1}$, we can rewrite the above as
\begin{align*}
u_t^\star &= u_{t-1} - g(0)^{-1}(y_t - y_\mathrm{min}) - g(0)^{-1}(y^\mathrm{nat}_{t+1} - y^\mathrm{nat}_t) \\
&\quad + g(0)^{-1}\left(\sum_{k=1}^{t} g(k) (u_{t-k-1} - u_{t-k})\right) 
\end{align*}
Since $K_+\leq g(0)^{-1}$ and $K_-\geq  g(0)^{-1}$, we have that $$- g(0)^{-1}(y_t-y_\mathrm{min}) \leq - K_+ (y_t-y_\mathrm{min})_+ - K_-(y_t-y_\mathrm{min})_- $$
And therefore, if $u_t$ is given by (\ref{eq:simple_update_true}), then
\begin{align*}
u_t^\star &\leq u_t - g(0)^{-1}(y^\mathrm{nat}_{t+1} - y^\mathrm{nat}_t) \\
&\qquad+ g(0)^{-1}\left(\sum_{k=1}^{t} g(k) (u_{t-k-1} - u_{t-k})\right) 
\end{align*}
Multiplying both sides by $g(0)$ and then adding $\sum_{k=1}^{t} g(k) u_{t-k} + y^\mathrm{nat}_{t+1}$, we have
\begin{align*}
&y^\mathrm{nat}_{t+1} + g(0) u_t^\star + \sum_{k=1}^{t} g(k) u_{t-k}   \\
\leq& y^\mathrm{nat}_{t+1} + \sum_{k=0}^{t} g(k) u_{t-k} 
  - (y^\mathrm{nat}_{t+1} - y^\mathrm{nat}_t) \\
&\qquad\qquad + \sum_{k=1}^{t} g(k) (u_{t-k-1} - u_{t-k})
\end{align*}
Since by construction the LHS is equal to exactly $y_\mathrm{min}$ and by the dynamics equation $y^\mathrm{nat}_{t+1} + \sum_{k=0}^{t} g(k) u_{t-k} = y_t$, rearranging the above yields:
$$y_\mathrm{min} + (y^\mathrm{nat}_{t+1} - y^\mathrm{nat}_t) - \sum_{k=1}^{t} g(k) (u_{t-k-1} - u_{t-k}) \leq y_t$$
which is precisely the stated inequality.

\section{Generalization to non-linear opponent processes}\label{app:nonlin_gen}

Our definition of \eqref{eq:dynamics} implies a linear contribution of the dose to the observed well-being metric. One may wonder if we can accommodate the case where the opponent process varies with $u$, i.e. $y_t$ evolves according to the following generalized version of (\ref{eq:dynamics}):
\begin{equation}\label{eq:dynamics_generalized}
y_{t+1} = \sum_{k=0}^{t} g(k, u_{t-k}) + y^\mathrm{nat}_{t+1}
\end{equation}
In this subsection, we show we can, under a corresponding generalization of well-behavedness from Definition~\ref{def:well_behaved_opp_process}. The rate of change of each $g(t, \cdot)$ with respect to the dose $u$ will play an important role, so we introduce the following two shorthand notations:
\begin{equation*}\label{eq:rate_of_change}
 \partial_u^{-}(t) \doteq \inf_{u} \partial_u g(t, u) \quad \quad \partial_u^{+}(t) \doteq \sup_{u} \partial_u g(t, u)   
\end{equation*}

\begin{definition}[Generalized Opponent Process]\label{def:opp_process_generalized} We say that $g$ is a generalized opponent process if there exists a time $\tau_0$ such that:
\begin{enumerate}
    \item $\inf_u g(\tau, u) \geq 0$ and $\partial_u^{-}(t) \geq 0$ when $\tau<\tau_0$,
    \item $\sup_u g(\tau, u) \leq 0$ and $\partial_u^{+}(t) \leq 0$ when $\tau\geq \tau_0$.
\end{enumerate} 
\end{definition}

\begin{definition}[Generalized LPOP]\label{def:well_behaved_opp_process_genealized} A generalized opponent process $g$ is a generalized linearly progressing opponent process (G-LPOP) if there exists an $\alpha \in (0,1)$ such that:
\begin{enumerate}
    \item $\partial_u^{+}(t+1) \leq \alpha \cdot \partial_u^{-}(t)$ when $t < \tau_0 - 1$,
    \item $\partial_u^{-}(t+1) \geq \alpha \cdot \partial_u^{+}(t)$ when $t \geq \tau_0$.
\end{enumerate}
\end{definition}

In this case the proof of Theorem~\ref{thm:max_greedy_opt} changes a bit, relying on Mean Value Theorem. We give the statement and proof in Theorem~\ref{thm:max_greedy_opt_generalized} below.

\begin{theorem}\label{thm:max_greedy_opt_generalized}For any G-LPOP $g$, taking the maximally greedy dose at every time maintains $y_t \geq y_\mathrm{min}$ for all $t$ with minimal cumulative dose. \end{theorem}

\begin{proof} Consider a dosing sequence $u^\star_t$ that maintains $y_t^\star \geq y_\mathrm{min}$ for all $t$ with minimal cumulative dose. Assume it is not maximally greedy (otherwise we would be done). This means that there exists $t_{0}$ and $\epsilon > 0$ such that $u_{t_{0}}^\star - \epsilon$ would produce a new $y_{t_0+1}' \geq y_\mathrm{min}$. Consider the modified dosing schedule
$$u_t'= \begin{cases} u_{t_0}^\star - \epsilon,  t=t_0 \\ u_{t_0+1}^\star + \alpha \epsilon, t=t_0+1 \\ u_t^\star \text{ otherwise}\end{cases}$$

and let $y_t'$ be the resulting observations. Since $u_t' = u_t^\star$ for $t < t_0$, we have that $ y_t' = y_t^\star \geq y_\mathrm{min}$ up to $t_0$. Plugging in the effect of the dose modifications at $t_0$ and $t_0 + 1$ into the generalized dynamics equation \eqref{eq:dynamics_generalized}, by the Mean Value Theorem, we get:
\begin{align*}
y_{t+1}' &= y_{t+1}^\star - g(t-t_0, u_{t_0}^\star) + g(t-t_0, u_{t_0}^\star - \epsilon) \\
&\quad \quad - g(t-t_0 - 1, u_{t_0+1}^\star) + g(t-t_0 - 1, u_{t_0+1}^\star + \alpha \epsilon)\\
&= y_{t+1}^\star - \epsilon \cdot \partial_u g(t-t_0, u_{t_0}^\star - \delta_0 \epsilon) \\
& \quad \quad + \alpha \epsilon \cdot \partial_u g(t-t_0 - 1, u_{t_0+1}^\star + \delta_1 \epsilon)
\end{align*}
for some $\delta_0 \in (0,1) , \delta_1 \in (0, \alpha)$. Since $g$ is a G-LPOP, this final expression is greater than or equal to $y_{t+1}^\star$ and hence is also greater than or equal to $y_\mathrm{min}$.
Since this holds for all $t > t_0$, we have that $y_t'$ satisfies the constraints for all $t$. To finalize, note that $y_t'$ is achieved with smaller cumulative dose since
$$\sum_{t=1}^T u_t'=\sum_{t=1}^T u_t^\star - \epsilon + \alpha \epsilon < \sum_{t=1}^T u_t^\star,$$
This is a contradiction, completing the proof.
\end{proof}

\noindent Instantiating this in \eqref{eq:opt_ynat} can be done by solving for the dose which exactly matches the setpoint:

\begin{proposition}\label{thm:opt_clairvoyant_generalized} For any G-LPOP $g$, assume we know a lower bound $y_t^\mathrm{nat, lb}$ such that $y_t^\mathrm{nat, lb} \leq y_t^\mathrm{nat}$ for all t. Then, for any instance of $y_t^\mathrm{nat}$ that is greater than the prescribed lower bound, the dosing sequence which takes $u_t$ be the solution to
\begin{equation}\label{eq:opt_generalized}
g(0, u) = \max\left\{0, y_\mathrm{min} - y_{t+1}^\mathrm{nat, lb} - \sum_{k=1}^t g(k, u_{t-k})\right\}
\end{equation}
maintains $y_t\geq y_\mathrm{min}$ for all $t$ with minimal cumulative dose.
\end{proposition}

Since $g(0, u)$ is increasing in $u$, we have the following guarantee:

\begin{remark} An $\epsilon$-accurate solution to (\ref{eq:opt_generalized}) can be computed in $\mathcal{O}(\log{\frac{1}{\epsilon}})$ steps.
\end{remark}

The above effectively shows that optimal tapering is equivalent to setpoint matching under non-linear $g$ as well. While it is not as clean to derive precise guarantees for the integral controller as in the main text, the result of Theorem~\ref{thm:max_greedy_opt_generalized}, combined with the robustness to $g(0)$ errors observed experimentally, strongly motivates the same approach in this more general setting.

\section{Additional Experiments}\label{app:exps}

\paragraph{Ablation over $g(0)$ range.}In Figure~\ref{fig:g0_range_ablation},  we plot the behavior of the integral protocol for the gains induced by different $g(0)$ lower/upper bounds against that of the optimal protocol and the non-adaptive baselines.\footnote{For the sake of readability we only plot the better baseline protocol for each opponent process (in black), i.e. exponential for $g_{A}/g_{B}/g_{C}$ and linear for $g_D$.} We see that the trends in Figure~\ref{fig:tapering_results} continue to hold over wider misspecification on the range of $g(0)$. In fact, the experiments suggest an overly conservative upper bound on $g(0)$ (leading to a smaller $K_-$) may improve performance in certain scenarios, and that setting $K_- = K_+$ may sometimes lead to higher constraint violation in the absence of padding. Overall, we observe that the protocol can be instantiated with surprisingly rough knowledge of the immediate effect on well-being of a medication without suffering significantly in terms of performance.

\begin{figure}[H]
\centering
\includegraphics[width=\linewidth]{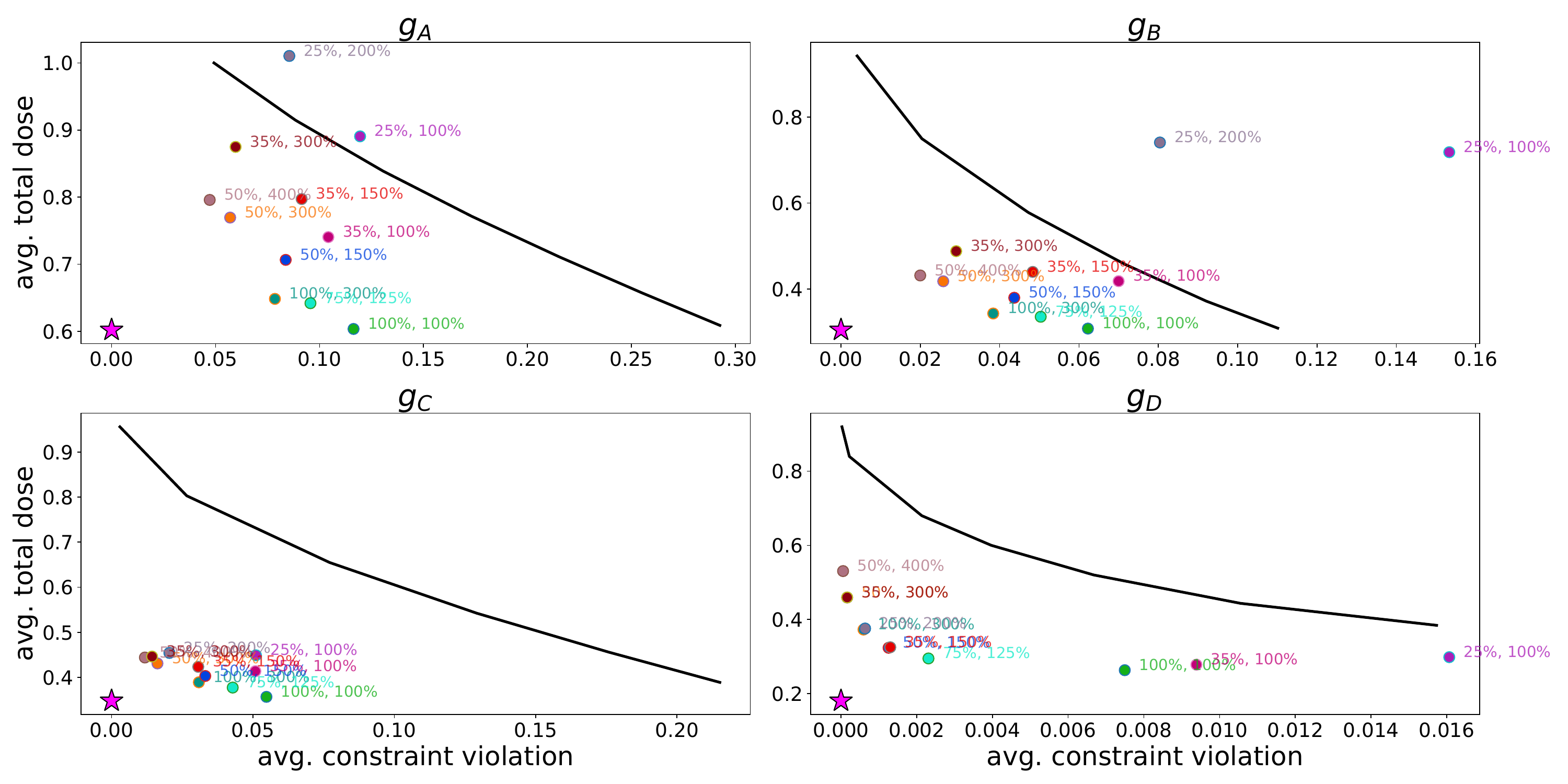}
\caption{Average cumulative constraint violation (x-axis) against average cumulative dose (y-axis) for different ranges of lower and upper bounds on $g(0)$. A label of $(p_1, p_2)$ in the legend corresponds to setting $K_-$ to $(p_1 g(0))^{-1}$ and $K_+$ to $(p_2g(0))^{-1}$.}
\label{fig:g0_range_ablation}
\end{figure}

\paragraph{Fraction of subjects fully tapered.}In Figure~\ref{fig:tapering_success}, we plot the fraction of units that can be fully tapered within $T_{\mathrm{taper}}$ steps using the integral protocol (instantiated with the gains and padding specifications from Section~\ref{sec:experiments}). We see that our method fully tapers a significant portion of the population within some pre-specified timeline with minimal constraint violation. This is in contrast to the linear and exponential protocols which induce a $0\%$ vs. $100\%$ tapering success step function given a time frame $T_\mathrm{taper}$. Across all settings considered, the baseline linear/exponential protocols would have to incur significantly larger average constraint violation (\textit{well} beyond the range covered by our plot) to fully taper individuals within the desired time frame. Finally, we again see a gap between the tapering success of the optimal and the integral tapering schemes, and that one way to address this gap is by using negative padding to allow more constraint violation (that however is still significantly smaller than that required for full tapering via non-adaptive methods).

\begin{figure}[H]
\centering
\includegraphics[width=\linewidth]{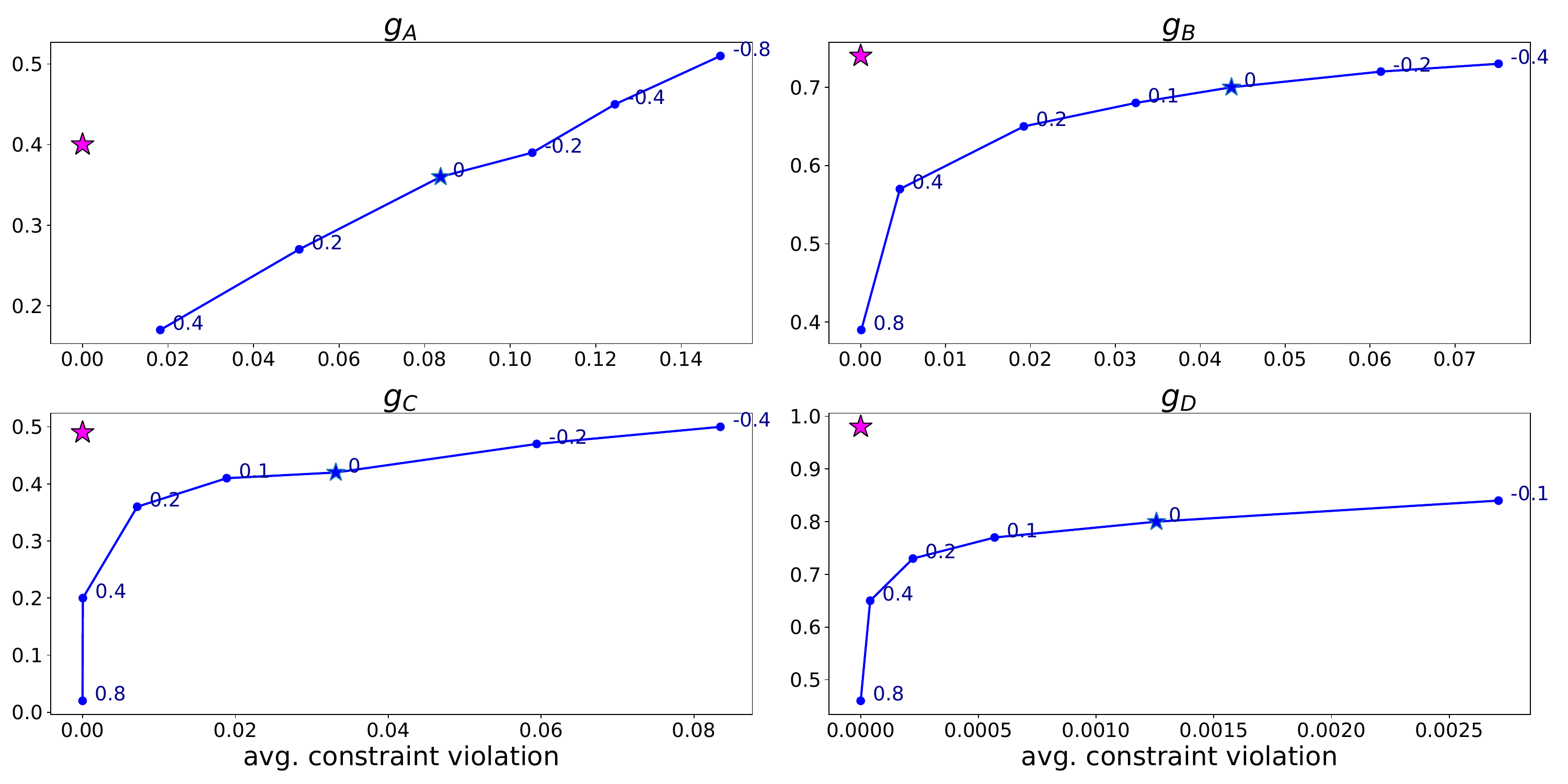}
\caption{Average constraint violation (x-axis) against the fraction of units fully tapered (y-axis) for the integral (\textcolor{blue}{blue}) and optimal (\textcolor{magenta}{$\star$}) protocols.}
\label{fig:tapering_success}
\end{figure}

\end{document}